\documentclass[11pt]{article}

\usepackage[margin=1in]{geometry}
\usepackage{amsmath,amssymb,amsthm,mathtools,mathrsfs,graphtex3.41}
\usepackage{hyperref}
\usepackage{enumitem}
\usepackage{graphicx}
\usepackage{booktabs}
\usepackage{microtype}
\usepackage{placeins}
\hypersetup{hidelinks}
\newtheorem{theorem}{Theorem}[section]
\newtheorem{definition}[theorem]{Definition}
\newtheorem{lemma}[theorem]{Lemma}
\newtheorem{corollary}[theorem]{Corollary}
\newtheorem{proposition}[theorem]{Proposition}
\newtheorem{remark}[theorem]{Remark}

\newcommand\Cy[1] {\ensuremath \mathscr C_{#1}}

\title{The Pentagon Graph Operator}
\author{Severino V.~Gervacio${}^{0000-0001-6036-5835}$\\
De La Salle University\\
0922 Manila, Philippines\\
~\\
Hiroshi Maehara\\
Ryukyu University\\
903-0213 Okinawa, Japan\\
~\\
Phoebe Chloe F.~Ramos\\
Ilocos Sur Polytechnic State College\\
2714 Ilocos Sur, Philippines}

\date{}

\begin{document}
\maketitle
\begin{abstract}
The pentagon graph $\Cy5(G)$ of a finite simple graph $G$ is the edge-intersection graph of its induced cycles of length five. We prove that the dodecahedral graph is mapped to the icosahedral graph $I$, which is a fixed point of the operator. Our principal result classifies every graph obtained by adjoining one vertex to $I$: it is eventually fixed if the added vertex belongs to no induced pentagon, and it is expanding otherwise. The proof shows that every new pentagon gives rise to a tadpole hat within at most two iterations. More generally, an induced icosahedral core with $h$ tadpole hats gives the lower bound $12+h2^k$ on the orders of the iterates. We also establish a replication principle and construct two vertex-disjoint induced copies of the Shrikhande graph in its pentagon graph, yielding the bound $16\cdot2^k$. The general vanishing--periodic--expanding trichotomy is included as background.


\end{abstract}

\section{Introduction}
A graph operator assigns a graph to each graph and preserves isomorphism. The line, triangle, and cycle graph operators provide classical examples. The cycle graph construction associates vertices with induced cycles and joins two of them when they share an edge; see \cite{Egawa,gervacio,seema}. Restricting the cycle length to three gives the triangle graph operator \cite{Balakrishnan,EgawaRamos}, while restricting it to four gives the quadrangle graph operator \cite{quadrangle}. Here we study the corresponding operator for induced cycles of length five.

The vanishing--periodic--expanding trichotomy follows from the general theorem on graph operators~\cite{fundamental}. Our aim is to determine how particular graph structures produce these behaviors. Changing the permitted cycle length changes which subgraphs survive the operator, so the trichotomy alone does not determine the behavior of any given graph. For example, all bipartite graphs vanish immediately under the pentagon operator, whereas the icosahedral graph is a fixed point.

Our main result gives a complete classification of one-vertex extensions of the icosahedron. Such an extension is eventually fixed precisely when the added vertex lies in no induced pentagon; every other extension is expanding. Thus the appearance of a new pentagon provides an exact criterion in this family. The proof identifies two possible core neighborhoods for a new pentagon and shows that each leads to a tadpole hat within two iterations. Once a tadpole hat appears, repeated doubling yields an explicit exponential lower bound.

A second construction uses the Shrikhande graph. We find two vertex-disjoint induced copies of this graph in its pentagon graph by translating two explicit pentagons in $\mathbb Z_4^2$. A general replication principle then gives exponential growth. These constructions exhibit two different sufficient mechanisms for expansion: multiplication of attachments around an induced fixed core, and multiplication of vertex-disjoint induced copies of a graph. The latter argument applies to induced-cycle operators of other lengths as well; the specific pentagon constructions are established here in detail.

\section{Definitions and induced subgraphs}
All graphs in this paper are finite, simple, and undirected. The graph with no vertices is denoted by $\varnothing$. We write $G[U]$ for the subgraph induced by $U\subseteq V(G)$ and $N_G(x)$ for the neighborhood of $x$. A cycle is identified with its subgraph, so cyclic shifts and reversal of a cyclic listing do not give new cycles. A \emph{pentagon} means an induced cycle of length five.

\begin{definition}
The \emph{pentagon graph} $\Cy5(G)$ has the pentagons of $G$ as its vertices. Distinct pentagons are adjacent if and only if they share an edge of $G$.
\end{definition}
We set $\Cy5(G)=\varnothing$ when $G$ has no pentagons, and put $\Cy5^0(G)=G$ and $\Cy5^{k+1}(G)=\Cy5(\Cy5^k(G))$. In particular, $\Cy5(\varnothing)=\varnothing$. The iterates form the sequence
\[
G,\ \Cy5(G),\ \Cy5^2(G),\ \ldots.
\]

\begin{definition}
A graph $G$ is \emph{pentagon-vanishing} if some iterate is empty. It is \emph{pentagon-periodic} if all its iterates are nonempty and there exist integers $m\ge0$ and $p\ge1$ such that
\[
\Cy5^{n+p}(G)\cong\Cy5^n(G)\qquad(n\ge m).
\]
The least such $m$ is the \emph{tail length}; the least eventual period is the \emph{period}. A graph with tail length zero and period one is a \emph{fixed point}. Finally, $G$ is \emph{pentagon-expanding} if
\[
|V(\Cy5^k(G))|\longrightarrow\infty\qquad(k\longrightarrow\infty).
\]
\end{definition}
The nonemptiness requirement in the periodic case separates periodicity from eventual disappearance.

\begin{lemma}\label{lem:induced}
If $H$ is an induced subgraph of $G$, then $\Cy5(H)$ is naturally an induced subgraph of $\Cy5(G)$. Consequently, $\Cy5^k(H)$ is isomorphic to an induced subgraph of $\Cy5^k(G)$ for every $k\ge0$.
\end{lemma}
\begin{proof}
A pentagon of $H$ remains induced in $G$, since the edges between its vertices are unchanged. Two such pentagons share an edge in $H$ exactly when they share an edge in $G$. The first assertion follows, and induction gives the second.
\end{proof}

\begin{proposition}[Replication principle]\label{prop:replication}
Let $H$ be a nonempty graph and let $r\ge2$ be an integer. Suppose that $\Cy5(H)$ contains $r$ pairwise vertex-disjoint induced subgraphs isomorphic to $H$. If $G$ contains an induced copy of $H$, then $\Cy5^k(G)$ contains at least $r^k$ pairwise vertex-disjoint induced copies of $H$ for every $k\ge0$. Consequently,
\[
|V(\Cy5^k(G))|\ge |V(H)|r^k\qquad(k\ge0),
\]
and $G$ is pentagon-expanding.
\end{proposition}
\begin{proof}
The assertion about induced copies holds for $k=0$. Suppose that $\Cy5^k(G)$ contains $r^k$ pairwise vertex-disjoint induced copies $H_1,\ldots,H_{r^k}$ of $H$. By Lemma~\ref{lem:induced}, each $\Cy5(H_i)$ is an induced subgraph of $\Cy5^{k+1}(G)$. These subgraphs are vertex-disjoint: a pentagon contained in one $H_i$ cannot also be contained in another, since their underlying vertex sets are disjoint. Each $\Cy5(H_i)$ contains $r$ pairwise vertex-disjoint induced copies of $H$, so the next iterate contains at least $r^{k+1}$ such copies. Induction proves the assertion, and counting their vertices gives the bound.
\end{proof}

Edges between the copies are allowed; their union need not be an induced disjoint union. The proof uses only induced-subgraph preservation and disjointness of the vertex sets supporting the cycles. It therefore also applies to the operator defined by edge intersections of induced cycles of any fixed length $\ell\ge3$.

\section{Vanishing graphs and a polyhedral fixed point}
Graphs with no induced pentagons vanish after one application of the operator. This includes all bipartite graphs. The cycle $C_5$ gives the sequence
\[
C_5,\quad K_1,\quad\varnothing,\quad\varnothing,\ldots.
\]
Here $C_5$ denotes the ordinary cycle graph; $\Cy5$ denotes the operator.

\subsection{The dodecahedral graph}
Let $D$ and $I$ denote the dodecahedral and icosahedral graphs, respectively.
\begin{proposition}\label{prop:dodeca}
The graph $D$ has exactly twelve pentagons, and $\Cy5(D)\cong I$.
\end{proposition}
\begin{proof}
Use the usual generalized Petersen representation of $D$, with vertices $r_i,s_i$ and edges $r_ir_{i+1}$, $r_is_i$, and $s_is_{i+2}$, where subscripts are taken modulo $10$. A pentagon uses an even number of spokes $r_is_i$. Four spokes would require at least four intervening nonspoke edges, so a pentagon has either zero or two spokes.

With zero spokes, it is one of the two inner cycles on the even or odd subscripts. With two spokes, the outer and inner paths have positive lengths whose sum is three. An outer path of length one cannot be joined by an inner path of length two, since the changes in subscripts are respectively $\pm1$ and $0,\pm4$. Thus the outer path has length two and the inner path has length one. The remaining ten pentagons are
\[
[r_i,r_{i+1},r_{i+2},s_{i+2},s_i,r_i]\qquad(i\in\mathbb Z_{10}).
\]
Each listed cycle is induced. These are precisely the twelve faces in the usual plane embedding of $D$. Their edge-intersection graph is its polyhedral dual, namely $I$.
\end{proof}

\subsection{An explicit fixed-point isomorphism}
We use the labeling of $I$ in Figure~\ref{fig:Icosa-Icosa}. For a concise description, put $u_i=i+1$ and $w_i=i+6$ for $i\in\mathbb Z_5$. Its edges are
\begin{equation}\label{eq:icosa}
au_i,\quad bw_i,\quad u_iu_{i+1},\quad w_iw_{i+1},\quad u_iw_i,\quad u_iw_{i+1}
\qquad(i\in\mathbb Z_5).
\end{equation}

\begin{proposition}\label{prop:fixed}
The graph $I$ has exactly twelve pentagons and satisfies $\Cy5(I)\cong I$. Thus $I$ is a fixed point, and $D$ has tail length one and period one.
\end{proposition}
\begin{proof}
Consider a pentagon containing $a$. Its two neighbors on the cycle must be nonadjacent vertices $u_i,u_{i+2}$. Its remaining two vertices cannot be upper-ring vertices, since these are adjacent to $a$. Equation~\eqref{eq:icosa} therefore forces the cycle
\[
[a,u_i,w_{i+1},w_{i+2},u_{i+2},a].
\]
There are exactly five such pentagons. By vertex-transitivity of the icosahedron, every vertex belongs to exactly five pentagons. Counting vertex--pentagon incidences gives $12\cdot5/5=12$ pentagons.

Each neighborhood in $I$ induces a pentagon. The twelve distinct cycles in Table~\ref{tab:cycles} are neighborhoods and hence give the complete list. Comparing their edge sets using~\eqref{eq:icosa} gives
\[
E(H_x)\cap E(H_y)\ne\varnothing
\quad\Longleftrightarrow\quad xy\in E(I)\qquad(x\ne y).
\]
For completeness, the neighbor sets on the cycle side are exactly
\begin{align*}
N(H_a)&=\{H_{u_i}:i\in\mathbb Z_5\},&
N(H_b)&=\{H_{w_i}:i\in\mathbb Z_5\},\\
N(H_{u_i})&=\{H_a,H_{u_{i-1}},H_{u_{i+1}},H_{w_i},H_{w_{i+1}}\},\\
N(H_{w_i})&=\{H_b,H_{w_{i-1}},H_{w_{i+1}},H_{u_{i-1}},H_{u_i}\}.
\end{align*}
Thus $x\mapsto H_x$ is an isomorphism. Since $D$ and $I$ have different orders, Proposition~\ref{prop:dodeca} gives the asserted tail length for $D$.
\end{proof}

\begin{table}[tbp]
\centering
\caption{The pentagons $H_x$ giving the isomorphism $x\mapsto H_x$. Each row lists a cyclic order, with the last vertex joined to the first.}\label{tab:cycles}
\begin{tabular}{ccl@{\qquad}ccl}
\toprule
$x$ & Center & Cyclic order of $H_x$ & $x$ & Center & Cyclic order of $H_x$\\
\midrule
$a$ & $a$ & $1,2,3,4,5$ & $b$ & $b$ & $6,7,8,9,10$\\
$1$ & $6$ & $1,7,b,10,5$ & $6$ & $2$ & $a,1,7,8,3$\\
$2$ & $8$ & $2,3,9,b,7$ & $7$ & $4$ & $a,3,9,10,5$\\
$3$ & $10$ & $4,5,6,b,9$ & $8$ & $1$ & $a,2,7,6,5$\\
$4$ & $7$ & $1,2,8,b,6$ & $9$ & $3$ & $a,2,8,9,4$\\
$5$ & $9$ & $3,4,10,b,8$ & $10$ & $5$ & $a,1,6,10,4$\\
\bottomrule
\end{tabular}
\end{table}
The center column identifies the vertex whose neighborhood induces $H_x$; it is generally different from $x$. In particular, the neighborhood map $x\mapsto I[N_I(x)]$ itself is not the isomorphism used here.

\begin{figure}[tbp]
\centering
\resizebox{0.98\textwidth}{!}{$\pic
\SetUnits[cm] (3.8,3.8,3.8)
\ViewPoint(27,2,5)
\ALIGN[c] ($I$) (0,0,-1.25)
\VertexRadius8pt
\PlotSize2
\PATH (-0.85,0,0.425) (-0.263,0.809,0.425) (0.688,0.5,0.425) (0.688,-0.5,0.425) (-0.263,-0.809,0.425)
\EDGE(-0.263,-0.809,0.425) (-0.85,0,0.425)
\PlotSize4
\EDGE(-0.263,-0.809,0.425) (0.688,-0.5,0.425)
(0.688,-0.5,0.425) (0.688,0.5,0.425)
(0.688,0.5,0.425) (-0.263,0.809,0.425)
\ALIGN[c] (5) (-0.85,0,0.425)
\ALIGN[c] (1) (-0.263,0.809,0.425)
\ALIGN[c] (2) (0.688,0.5,0.425)
\ALIGN[c] (3) (0.688,-0.5,0.435)
\ALIGN[c] (4) (-0.263,-0.809,0.425)
\PlotSize2
\PATH(0.85,0,-0.425) (0.263,-0.809, -0.425) (-0.688,-0.5,-0.425) (-0.688,0.5,-0.425) (0.263,0.809,-0.425)
\EDGE(0.263,0.809,-0.425) (0.85,0,-0.425)
\PlotSize4
\EDGE(0.263,-0.809,-0.425) (0.85,0,-0.425)
(0.85,0,-0.425) (0.263,0.809,-0.425)
\ALIGN[c] (8) (0.85,0,-0.425)
\ALIGN[c] (9) (0.263,-0.809,-0.425)
\ALIGN[c] (10) (-0.688,-0.5,-0.425)
\ALIGN[c] (6) (-0.688,0.5,-0.425)
\ALIGN[c] (7) (0.263,0.809,-0.425)
\PlotSize2
\VERTEX(0,0,0.951)
\EDGE(0,0,0.951) (-0.85,0,0.425) (0,0,0.951) (-0.263,0.809,0.425) (0,0,0.951) (0.688,0.5,0.425) (0,0,0.951) (0.688,-0.5,0.425) (0,0,0.951) (-0.263,-0.809,0.425)
\PlotSize4
\EDGE(0,0,0.951) (-0.263,-0.809,0.425)
(0,0,0.951) (0.688,-0.5,0.425)
(0,0,0.951) (0.688,0.5,0.425)
(0,0,0.951) (-0.263,0.809,0.425)
\PlotSize2
\VERTEX(0,0,-0.951)
\EDGE(0,0,-0.951) (0.85,0,-0.425) (0,0,-0.951) (0.263,-0.809,-0.425) (0,0,-0.951) (-0.688,-0.5,-0.425) (0,0,-0.951) (-0.688, 0.5,-0.425) (0,0,-0.951) (0.263,0.809,-0.425)
\PlotSize4
\EDGE(0,0,-0.951) (0.263,-0.809,-0.425) 
(0,0,-0.951) (0.85,0,-0.425)
(0,0,-0.951) (0.263,0.809,-0.425)
\PlotSize2
\PATH(-0.688,0.5,-0.425)  (-0.85,0,0.425)
 (-0.688,-0.5,-0.425)  (-0.263,-0.809,0.425) (0.263,-0.809,-0.425)  (0.688,-0.5,0.425) (0.85,0,-0.425) (0.688,0.5,0.425) (0.263,0.809,-0.425) (-0.263,0.809,0.425) (-0.688,0.5,-0.425)
\PlotSize4
\EDGE(-0.263,-0.809,0.425) (0.263,-0.809,-0.425)
(0.263,-0.809,-0.425) (0.688,-0.5,0.425)
(0.688,-0.5,0.425) (0.85,0,-0.425)
(0.85,0,-0.425) (0.688,0.5,0.425)
(0.688,0.5,0.425) (0.263,0.809,-0.425)
(0.263,0.809,-0.425) (-0.263,0.809,0.425)
\PlotSize2
\ALIGN[c] ($a$) (0,0,0.951)
\ALIGN[c] ($b$) (0,0,-0.951)
\Translate(0,0) (2,0)
\ALIGN[c] ($\Cy5(I)$) (0,0,-1.25) 
\VertexRadius9pt
\SimpleVertex{black}
\PlotSize2
\PATH (-0.85,0,0.425) (-0.263,0.809,0.425) (0.688,0.5,0.425) (0.688,-0.5,0.425) (-0.263,-0.809,0.425)
\EDGE(-0.263,-0.809,0.425) (-0.85,0,0.425)
\PlotSize4
\EDGE(-0.263,-0.809,0.425) (0.688,-0.5,0.425)
(0.688,-0.5,0.425) (0.688,0.5,0.425)
(0.688,0.5,0.425) (-0.263,0.809,0.425)
\ALIGN[c] ($H_{1}$) (-0.263,0.809,0.425)
\ALIGN[c] ($H_{2}$) (0.688,0.5,0.425)
\ALIGN[c] ($H_{3}$) (0.688,-0.5,0.435)
\ALIGN[c] ($H_{4}$) (-0.263,-0.809,0.425)
\ALIGN[c] ($H_{5}$) (-0.85,0,0.425)
\PlotSize2
\PATH(0.85,0,-0.425) (0.263,-0.809, -0.425) (-0.688,-0.5,-0.425) (-0.688,0.5,-0.425) (0.263,0.809,-0.425)
\EDGE(0.263,0.809,-0.425) (0.85,0,-0.425)
\PlotSize4
\EDGE(0.263,-0.809,-0.425) (0.85,0,-0.425)
(0.85,0,-0.425) (0.263,0.809,-0.425)
\ALIGN[c] ($H_{8}$) (0.85,0,-0.425)
\ALIGN[c] ($H_{9}$) (0.263,-0.809,-0.425)
\ALIGN[c] ($H_{10}$) (-0.688,-0.5,-0.425)
\ALIGN[c] ($H_{6}$) (-0.688,0.5,-0.425)
\ALIGN[c] ($H_{7}$) (0.263,0.809,-0.425)
\PlotSize2
\VERTEX(0,0,0.951)
\EDGE(0,0,0.951) (-0.85,0,0.425) (0,0,0.951) (-0.263,0.809,0.425) (0,0,0.951) (0.688,0.5,0.425) (0,0,0.951) (0.688,-0.5,0.425) (0,0,0.951) (-0.263,-0.809,0.425)
\PlotSize4
\EDGE(0,0,0.951) (-0.263,-0.809,0.425)
(0,0,0.951) (0.688,-0.5,0.425)
(0,0,0.951) (0.688,0.5,0.425)
(0,0,0.951) (-0.263,0.809,0.425)
\PlotSize2
\VERTEX(0,0,-0.951)
\EDGE(0,0,-0.951) (0.85,0,-0.425) (0,0,-0.951) (0.263,-0.809,-0.425) (0,0,-0.951) (-0.688,-0.5,-0.425) (0,0,-0.951) (-0.688, 0.5,-0.425) (0,0,-0.951) (0.263,0.809,-0.425)
\PlotSize4
\EDGE(0,0,-0.951) (0.263,-0.809,-0.425) 
(0,0,-0.951) (0.85,0,-0.425)
(0,0,-0.951) (0.263,0.809,-0.425)
\PlotSize2
\PATH(-0.688,0.5,-0.425)  (-0.85,0,0.425)
 (-0.688,-0.5,-0.425)  (-0.263,-0.809,0.425) (0.263,-0.809,-0.425)  (0.688,-0.5,0.425) (0.85,0,-0.425) (0.688,0.5,0.425) (0.263,0.809,-0.425) (-0.263,0.809,0.425) (-0.688,0.5,-0.425)
\PlotSize4
\EDGE(-0.263,-0.809,0.425) (0.263,-0.809,-0.425)
(0.263,-0.809,-0.425) (0.688,-0.5,0.425)
(0.688,-0.5,0.425) (0.85,0,-0.425)
(0.85,0,-0.425) (0.688,0.5,0.425)
(0.688,0.5,0.425) (0.263,0.809,-0.425)
(0.263,0.809,-0.425) (-0.263,0.809,0.425)
\PlotSize2
\ALIGN[c] ($H_{a}$) (0,0,0.951)
\ALIGN[c] ($H_{b}$) (0,0,-0.951)
\cip$}
\caption{The isomorphism $x\mapsto H_x$ from $I$ to $\Cy5(I)$. The cycles $H_x$ are defined in Table~\ref{tab:cycles}.}
\label{fig:Icosa-Icosa}
\end{figure}
\FloatBarrier

\section{Expansion from tadpole hats}
A \emph{tadpole} $T_{3,1}$ consists of a triangle and a pendant vertex attached to one vertex of the triangle. Let $J$ be an induced copy of $I$ in a graph $G$. A vertex $v\notin V(J)$ is a \emph{tadpole hat relative to $J$} if its neighborhood within $J$ consists of exactly the four vertices of an induced tadpole. No restriction is imposed on its neighbors outside $J$.

\begin{lemma}\label{lem:transitive}
The automorphism group of $I$ acts transitively on its induced tadpoles.
\end{lemma}
\begin{proof}
Every triangle of the icosahedron is a face. For a triangular face $xyz$ and a distinguished vertex $x$, there is exactly one neighbor of $x$ adjacent to neither $y$ nor $z$. Indeed, the neighbors of $x$ induce a $5$-cycle, in which $y$ and $z$ are consecutive; exactly one of its other vertices is adjacent to neither. Thus an induced tadpole is determined by a face and its distinguished attachment vertex. Rotational symmetries of the regular icosahedron act transitively on faces, and rotations about a face center act transitively on its three vertices. These symmetries therefore act transitively on induced tadpoles.
\end{proof}

\begin{lemma}\label{lem:local}
Let $I_1$ be obtained from $I$ by adding one tadpole hat $v$. There are exactly two pentagons containing $v$. Each is a tadpole hat relative to the induced core $\Cy5(I)$ in $\Cy5(I_1)$.
\end{lemma}
\begin{proof}
By Lemma~\ref{lem:transitive}, we may assume that
\[
N_{I_1}(v)=\{3,4,9,b\}.
\]
The vertices $3,4,9$ form a triangle, and $b$ is pendant at $9$. The two neighbors of $v$ on a pentagon must be nonadjacent, so they are either $4,b$ or $3,b$. The other two vertices must lie outside $N_{I_1}(v)$.

For the pair $4,b$, equation~\eqref{eq:icosa} leaves $5$ or $a$ as the next vertex after $4$, and $6,7,8$ as the vertex before $b$ (the choice $10$ would create a chord to $4$). Of these choices, the only edge joining the two lists is $5\,6$. For the pair $3,b$, the first list is $2,a$ and the second list is $6,7,10$ (excluding $8$, which is adjacent to $3$); the only joining edge is $2\,7$. Both resulting cycles are induced. They are
\[
p=[v,4,5,6,b,v],\qquad q=[v,3,2,7,b,v].
\]
Using Table~\ref{tab:cycles}, their neighbors \emph{within the core} are
\[
N(p)\cap V(\Cy5(I))=\{H_a,H_4,H_3,H_8\},\qquad
N(q)\cap V(\Cy5(I))=\{H_a,H_1,H_2,H_8\}.
\]
In either displayed order, the induced edges are $12,13,23,34$. Thus each neighborhood induces a tadpole. The cycles $p$ and $q$ also share the edge $vb$, so they are adjacent to each other; this additional edge does not affect either core neighborhood.
\end{proof}

\begin{figure}[tbp]
\centering
\resizebox{0.98\textwidth}{!}{$\pic
\SetUnits[cm] (3.5,3.5,3.5)
\ViewPoint(27,2,5)
\ALIGN[c] ({(a) $I_1$}) (0,0,-1.25)
\VertexRadius8pt
\PlotSize2
\PATH (-0.85,0,0.425) (-0.263,0.809,0.425) (0.688,0.5,0.425) (0.688,-0.5,0.425) (-0.263,-0.809,0.425)
\EDGE(-0.263,-0.809,0.425) (-0.85,0,0.425)
\PlotSize4
\EDGE(-0.263,-0.809,0.425) (0.688,-0.5,0.425)
(0.688,-0.5,0.425) (0.688,0.5,0.425)
(0.688,0.5,0.425) (-0.263,0.809,0.425)
\ALIGN[c] (5) (-0.85,0,0.425)
\ALIGN[c] (1) (-0.263,0.809,0.425)
\ALIGN[c] (2) (0.688,0.5,0.425)
\ALIGN[c] (3) (0.688,-0.5,0.435)
\ALIGN[c] (4) (-0.263,-0.809,0.425)
\PlotSize2
\PATH(0.85,0,-0.425) (0.263,-0.809, -0.425) (-0.688,-0.5,-0.425) (-0.688,0.5,-0.425) (0.263,0.809,-0.425)
\EDGE(0.263,0.809,-0.425) (0.85,0,-0.425)
\PlotSize4
\EDGE(0.263,-0.809,-0.425) (0.85,0,-0.425)
(0.85,0,-0.425) (0.263,0.809,-0.425)
\ALIGN[c] (8) (0.85,0,-0.425)
\ALIGN[c] (9) (0.263,-0.809,-0.425)
\ALIGN[c] (10) (-0.688,-0.5,-0.425)
\ALIGN[c] (6) (-0.688,0.5,-0.425)
\ALIGN[c] (7) (0.263,0.809,-0.425)
\PlotSize2
\VERTEX(0,0,0.951)
\EDGE(0,0,0.951) (-0.85,0,0.425) (0,0,0.951) (-0.263,0.809,0.425) (0,0,0.951) (0.688,0.5,0.425) (0,0,0.951) (0.688,-0.5,0.425) (0,0,0.951) (-0.263,-0.809,0.425)
\PlotSize4
\EDGE(0,0,0.951) (-0.263,-0.809,0.425)
(0,0,0.951) (0.688,-0.5,0.425)
(0,0,0.951) (0.688,0.5,0.425)
(0,0,0.951) (-0.263,0.809,0.425)
\PlotSize2
\VERTEX(0,0,-0.951)
\EDGE(0,0,-0.951) (0.85,0,-0.425) (0,0,-0.951) (0.263,-0.809,-0.425) (0,0,-0.951) (-0.688,-0.5,-0.425) (0,0,-0.951) (-0.688, 0.5,-0.425) (0,0,-0.951) (0.263,0.809,-0.425)
\PlotSize4
\EDGE(0,0,-0.951) (0.263,-0.809,-0.425) 
(0,0,-0.951) (0.85,0,-0.425)
(0,0,-0.951) (0.263,0.809,-0.425)
\PlotSize2
\PATH(-0.688,0.5,-0.425)  (-0.85,0,0.425)
 (-0.688,-0.5,-0.425)  (-0.263,-0.809,0.425) (0.263,-0.809,-0.425)  (0.688,-0.5,0.425) (0.85,0,-0.425) (0.688,0.5,0.425) (0.263,0.809,-0.425) (-0.263,0.809,0.425) (-0.688,0.5,-0.425)
\PlotSize4
\EDGE(-0.263,-0.809,0.425) (0.263,-0.809,-0.425)
(0.263,-0.809,-0.425) (0.688,-0.5,0.425)
(0.688,-0.5,0.425) (0.85,0,-0.425)
(0.85,0,-0.425) (0.688,0.5,0.425)
(0.688,0.5,0.425) (0.263,0.809,-0.425)
(0.263,0.809,-0.425) (-0.263,0.809,0.425)
\PlotSize2
\ALIGN[c] ($a$) (0,0,0.951)
\ALIGN[c] ($b$) (0,0,-0.951)
\VERTEX(0.4,-1,-0.6)
\ALIGN[c] ($v$) (0.4,-1,-0.6)
\ondashes
\EDGE(0.4,-1,-0.6)  (0.688,-0.5,0.435) (0.4,-1,-0.6) (-0.263,-0.809,0.425) (0.4,-1,-0.6) (0.263,-0.809,-0.425) (0.4,-1,-0.6) (0,0,-0.951)
\offdashes
\Translate(0,0) (2,0)
\ALIGN[c] ({(b) $\Cy5(I_1)$}) (0,0,-1.25) 
\VertexRadius9pt
\SimpleVertex{black}
\PlotSize2
\PATH (-0.85,0,0.425) (-0.263,0.809,0.425) (0.688,0.5,0.425) (0.688,-0.5,0.425) (-0.263,-0.809,0.425)
\EDGE(-0.263,-0.809,0.425) (-0.85,0,0.425)
\PlotSize4
\EDGE(-0.263,-0.809,0.425) (0.688,-0.5,0.425)
(0.688,-0.5,0.425) (0.688,0.5,0.425)
(0.688,0.5,0.425) (-0.263,0.809,0.425)
\ALIGN[c] ($H_{1}$) (-0.263,0.816,0.43)
\ALIGN[c] ($H_{2}$) (0.688,0.5,0.425)
\ALIGN[c] ($H_{3}$) (0.688,-0.5,0.435)
\ALIGN[c] ($H_{4}$) (-0.263,-0.809,0.428)
\ALIGN[c] ($H_{5}$) (-0.85,0,0.436)
\PlotSize2
\PATH(0.85,0,-0.425) (0.263,-0.809, -0.425) (-0.688,-0.5,-0.425) (-0.688,0.5,-0.425) (0.263,0.809,-0.425)
\EDGE(0.263,0.809,-0.425) (0.85,0,-0.425)
\PlotSize4
\EDGE(0.263,-0.809,-0.425) (0.85,0,-0.425)
(0.85,0,-0.425) (0.263,0.809,-0.425)
\ALIGN[c] ($H_{7}$) (0.263,0.809,-0.425) 
\ALIGN[c] ($H_{8}$) (0.85,0,-0.425) 
\ALIGN[c] ($H_{9}$) (0.263,-0.809,-0.425) 
\ALIGN[c] ($H_{10}$) (-0.688,-0.5,-0.425) 
\ALIGN[c] ($H_{6}$) (-0.688,0.5,-0.425) 
\PlotSize2
\VERTEX(0,0,0.951)
\EDGE(0,0,0.951) (-0.85,0,0.425) (0,0,0.951) (-0.263,0.809,0.425) (0,0,0.951) (0.688,0.5,0.425) (0,0,0.951) (0.688,-0.5,0.425) (0,0,0.951) (-0.263,-0.809,0.425)
\PlotSize4
\EDGE(0,0,0.951) (-0.263,-0.809,0.425)
(0,0,0.951) (0.688,-0.5,0.425)
(0,0,0.951) (0.688,0.5,0.425)
(0,0,0.951) (-0.263,0.809,0.425)
\PlotSize2
\VERTEX(0,0,-0.951)
\EDGE(0,0,-0.951) (0.85,0,-0.425) (0,0,-0.951) (0.263,-0.809,-0.425) (0,0,-0.951) (-0.688,-0.5,-0.425) (0,0,-0.951) (-0.688, 0.5,-0.425) (0,0,-0.951) (0.263,0.809,-0.425)
\PlotSize4
\EDGE(0,0,-0.951) (0.263,-0.809,-0.425) 
(0,0,-0.951) (0.85,0,-0.425)
(0,0,-0.951) (0.263,0.809,-0.425)
\PlotSize2
\PATH(-0.688,0.5,-0.425)  (-0.85,0,0.425)
 (-0.688,-0.5,-0.425)  (-0.263,-0.809,0.425) (0.263,-0.809,-0.425)  (0.688,-0.5,0.425) (0.85,0,-0.425) (0.688,0.5,0.425) (0.263,0.809,-0.425) (-0.263,0.809,0.425) (-0.688,0.5,-0.425)
\PlotSize4
\EDGE(-0.263,-0.809,0.425) (0.263,-0.809,-0.425)
(0.263,-0.809,-0.425) (0.688,-0.5,0.425)
(0.688,-0.5,0.425) (0.85,0,-0.425)
(0.85,0,-0.425) (0.688,0.5,0.425)
(0.688,0.5,0.425) (0.263,0.809,-0.425)
(0.263,0.809,-0.425) (-0.263,0.809,0.425)
\PlotSize2
\ALIGN[c] ($H_{a}$) (0,0,0.951)
\ALIGN[c] ($H_{b}$) (0,0,-0.951)
\SimpleVertex{black}
\VERTEX(0,-0.6,1.2) (0,0.6,1.2)
\ALIGN[c] ($p$) (0,-0.6,1.2)
\ondashes
\EDGE(0,-0.6,1.2) (0,0,0.951) (0,-0.6,1.2) (0.688,-0.5,0.435) (0,-0.6,1.2) (-0.263,-0.809,0.425) (0,-0.6,1.2) (0.85,0,-0.425)
\ALIGN[c] ($q$) (0,0.6,1.2)
\EDGE (0,0.6,1.2) (0,0,0.951) (0,0.6,1.2)  (0.85,0,-0.425) (0,0.6,1.2)  (0.688,0.5,0.425) (0,0.6,1.2) (-0.263,0.809,0.425)
\EDGE(0,-0.6,1.2) (0,0.6,1.2)
\cip$}
\caption{A tadpole hat $v$ produces two hats $p$ and $q$. The edge $pq$ is retained.}
\label{fig:hat}
\end{figure}
\FloatBarrier

\begin{theorem}\label{thm:hats}
Suppose a graph $G$ contains an induced icosahedral graph $J$ and $h\ge1$ distinct tadpole hats relative to $J$. Then
\[
|V(\Cy5^k(G))|\ge 12+h2^k\qquad(k\ge0).
\]
In particular, $G$ is pentagon-expanding. Edges between hats, and additional vertices outside $J$, are permitted.
\end{theorem}
\begin{proof}
By Lemma~\ref{lem:induced} and Proposition~\ref{prop:fixed}, the pentagons contained in $J$ induce a new icosahedral core $J'=\Cy5(J)$ in $\Cy5(G)$.

For each of the $h$ hats $v$, apply Lemma~\ref{lem:local} to the induced subgraph $G[V(J)\cup\{v\}]$. Its two pentagons containing $v$ remain induced in $G$. Each has exactly the same neighbors within $J'$ as in the local calculation, and hence is a tadpole hat relative to $J'$. Pentagons arising from different hats are distinct: each contains exactly one vertex outside $J$, namely the hat from which it was constructed. Consequently, $J'$ has at least $2h$ distinct hats.

This argument requires no condition on edges between hats. Repeating it proves by induction that $\Cy5^k(G)$ contains an induced icosahedral core and at least $h2^k$ hats. Counting these vertices proves the bound.
\end{proof}

\section{Classification of one-vertex extensions of the icosahedron}
We now allow the added vertex to have an arbitrary neighborhood in $I$. The tadpole-hat construction extends to an exact criterion for expansion in this entire family.

For each $x\in V(I)$, put $C_x=I[N_I(x)]$. These twelve pentagons are the vertices of the induced core $I'=\Cy5(I)$. The subscript $x$ denotes the neighborhood center, so $C_x$ is generally different from the cycle $H_x$ in Table~\ref{tab:cycles}. In particular, we use the edge-intersection relations between the $C_x$ rather than assuming that $x\mapsto C_x$ preserves adjacency.

\begin{lemma}\label{lem:path-types}
Up to an automorphism of $I$, every induced four-vertex path has one of the forms
\[
[a,1,6,b],\qquad [a,1,6,10].
\]
If an induced pentagon in a graph containing $I$ consists of such a path and one vertex outside $I$, then its neighbors within $I'$ induce a four-vertex path or a tadpole, respectively.
\end{lemma}
\begin{proof}
Vertex-transitivity lets us move the first vertex of the path to $a$. A rotation fixing $a$ then moves its second vertex to $1=u_0$. The third vertex must be adjacent to $1$ and nonadjacent to $a$, so it is $6=w_0$ or $7=w_1$. The automorphism fixing $a,b$ and sending
\[
u_i\longmapsto u_{-i},\qquad w_i\longmapsto w_{1-i}
\quad(i\in\mathbb Z_5)
\]
fixes $1$ and exchanges these two possibilities. We may therefore take the third vertex to be $6$. Of its neighbors, only $b$ and $10$ can be the fourth vertex without repetition or a chord. This proves the first assertion.

Let $R$ be the pentagon formed by adding an outside vertex to the endpoints of one of these paths. The two edges of $R$ incident with the outside vertex belong to no core pentagon. Thus its neighbors in $I'$ are determined entirely by the three path edges. Comparing with Table~\ref{tab:cycles} gives
\[
\begin{array}{c|c}
\text{Path in }I & N_{\Cy5(G)}(R)\cap V(I')\\ \hline
{[a,1,6,b]} & \{C_2,C_5,C_7,C_{10}\}\\
{[a,1,6,10]} & \{C_2,C_5,C_7,C_b\}.
\end{array}
\]
Here $G$ denotes the ambient graph. The first set induces precisely the path $C_2C_5C_7C_{10}$. In the second set, $C_2,C_5,C_b$ form a triangle and $C_7$ is pendant at $C_5$. These are the required two types.
\end{proof}

\begin{lemma}\label{lem:path-hat}
Suppose that a graph contains an induced icosahedral core $J$ and an outside vertex $w$ whose neighborhood within $J$ induces a four-vertex path. Then its pentagon graph contains a tadpole hat relative to the induced core $\Cy5(J)$.
\end{lemma}
\begin{proof}
Identify $J$ with the labeled graph $I$. By Lemma~\ref{lem:path-types}, we may assume that the neighbors of $w$ within $I$ are either
\[
\{a,1,6,b\}\quad\text{or}\quad\{a,1,6,10\}.
\]
In both cases,
\[
Q=[w,a,2,7,6,w]
\]
is an induced pentagon. Indeed, its consecutive pairs are edges, its internal path $[a,2,7,6]$ has no chord, and $w$ is adjacent to neither $2$ nor $7$. Comparison of the three core edges $a2$, $27$, and $76$ with the core pentagons gives exactly the neighbors
\[
C_1,\ C_3,\ C_8,\ C_b.
\]
The vertices $C_1,C_3,C_b$ form a triangle, and $C_8$ is pendant at $C_1$. Thus $Q$ is a tadpole hat. Additional vertices and edges outside the core and $w$ do not affect this calculation.
\end{proof}

\begin{theorem}[One-vertex extension classification]\label{thm:attachment}
Let $G$ be obtained from $I$ by adjoining one vertex $v$, with an arbitrary neighborhood in $I$.
\begin{enumerate}[label=(\roman*)]
\item If $v$ belongs to no induced pentagon of $G$, then $\Cy5(G)\cong I$. Thus $G$ is pentagon-periodic with tail length one and period one.
\item If $v$ belongs to an induced pentagon of $G$, then $G$ is pentagon-expanding. More precisely,
\[
|V(\Cy5^k(G))|\ge12+2^{k-2}\qquad(k\ge2).
\]
\end{enumerate}
These alternatives exhaust all one-vertex extensions of $I$.
\end{theorem}
\begin{proof}
If no pentagon contains $v$, all pentagons of $G$ lie in $I$, and their edge-intersection relations are unchanged. Hence $\Cy5(G)=\Cy5(I)\cong I$. Since $G$ has thirteen vertices and $I$ has twelve, its tail length is one.

Otherwise, choose an induced pentagon $R$ containing $v$. Its four core vertices induce a path. By Lemma~\ref{lem:path-types}, the neighborhood of $R$ within $I'=\Cy5(I)$ induces either a tadpole or a four-vertex path. In the first case, a tadpole hat is already present after one iteration. In the second case, Lemma~\ref{lem:path-hat} produces one after the next iteration. If a hat appears after one iteration, Lemma~\ref{lem:local} ensures that hats are still present after two iterations. Thus $\Cy5^2(G)$ always contains an induced icosahedral core with at least one tadpole hat. Applying Theorem~\ref{thm:hats} to this graph gives
\[
|V(\Cy5^{k}(G))|\ge12+2^{k-2}\qquad(k\ge2),
\]
which proves expansion.
\end{proof}

\begin{corollary}\label{cor:attachment-supergraph}
If a graph $G$ contains an induced icosahedral core $J$ and an induced pentagon with exactly one vertex outside $J$, then $G$ is pentagon-expanding and satisfies the bound in Theorem~\ref{thm:attachment}(ii).
\end{corollary}
\begin{proof}
Apply Theorem~\ref{thm:attachment} to the subgraph induced by $J$ and the single outside vertex, and then use Lemma~\ref{lem:induced}.
\end{proof}

\begin{remark}[Finite verification]
An independent enumeration of all $2^{12}=4096$ possible neighborhoods of $v$ gives $132$ choices for which no new pentagon occurs and $3964$ choices for which at least two new pentagons occur. No choice produces exactly one new pentagon. These counts concern subsets of the labeled vertex set of $I$, not isomorphism classes of extensions.

The enumeration can be reproduced directly from~\eqref{eq:icosa}: list every induced four-vertex path once up to reversal, and, for each subset $U\subseteq V(I)$, count those paths whose endpoints lie in $U$ and whose internal vertices lie outside $U$. Each counted path gives exactly one pentagon through a new vertex with neighborhood $U$. The classification theorem and its growth bound follow from the preceding proofs and do not depend on this enumeration.
\end{remark}

\section{Expansion of the Shrikhande graph}
The Shrikhande graph $S$, arising from the exceptional case in Shrikhande's study of the $L_2$ association scheme~\cite{shrikhande}, is strongly regular with parameters $(16,6,2,2)$. We represent $S$ on $\mathbb Z_4^2$, with two vertices adjacent when their difference belongs to
\[
A=\{\pm(1,0),\ \pm(0,1),\ \pm(1,1)\}.
\]
All coordinates in this section are taken modulo $4$.

\begin{theorem}\label{thm:shrikhande}
The graph $\Cy5(S)$ contains two vertex-disjoint induced copies of $S$. Consequently,
\[
|V(\Cy5^k(S))|\ge16\cdot2^k\qquad(k\ge0),
\]
and $S$ is pentagon-expanding.
\end{theorem}
\begin{proof}
Consider the pentagons with the following cyclic orders:
\begin{align*}
P&=[(0,0),(0,1),(0,2),(1,3),(1,0),(0,0)],\\
Q&=[(0,0),(0,1),(0,2),(3,2),(3,3),(0,0)].
\end{align*}
The five consecutive differences belong to $A$, and the five nonconsecutive differences do not, so both cycles are induced. For $R=P,Q$, define its translation family
\[
\mathscr R=\{R+t:t\in\mathbb Z_4^2\}.
\]
Each family has sixteen distinct members. Indeed, every nonzero translation has order two or four and has orbits of that length on $\mathbb Z_4^2$. A five-element set cannot be a union of such orbits.

For either $R=P$ or $R=Q$, comparison of its edge set with its translates gives, for $t\ne(0,0)$,
\begin{equation}\label{eq:translate}
E(R)\cap E(R+t)\ne\varnothing
\quad\Longleftrightarrow\quad
 t\in B:=\{\pm(0,1),\ \pm(1,2),\ \pm(1,3)\}.
\end{equation}
To verify this calculation directly, the initial vertices of the three edges directed by $(0,1)$ are
\[
\begin{array}{c|c}
R&\text{Initial vertices}\\ \hline
P&(0,0),(0,1),(1,3)\\
Q&(0,0),(0,1),(3,2).
\end{array}
\]
Their nonzero pairwise differences are exactly $B$. The remaining two edges have distinct directions, represented by $(1,0)$ and $(1,1)$, with only one edge in each direction. Neither edge is fixed as an unordered pair by a nonzero translation, since its endpoint difference has order four. Thus they give no additional values of $t$, proving~\eqref{eq:translate}.

The invertible linear map
\[
\phi(x,y)=(y,x+2y)
\]
sends $(1,0),(0,1),(1,1)$ to $(0,1),(1,2),(1,3)$, respectively. Hence $\phi(A)=B$, and
\[
(x,y)\longmapsto R+\phi(x,y)
\]
\nopagebreak
is an isomorphism from $S$ onto the subgraph induced by $\mathscr R$.

The two translation families are disjoint. In $P$, three vertices have first coordinate $0$ and two have first coordinate $1$; in $Q$, three have first coordinate $0$ and two have first coordinate $3$. The difference between the coordinate occurring twice and the coordinate occurring three times is invariant under translation. It is $1$ for $P$ and $-1$ for $Q$, which are distinct modulo $4$. Thus no translate of $P$ equals a translate of $Q$.

Applying Proposition~\ref{prop:replication} with $H=S$ and $r=2$ shows that $\Cy5^k(S)$ contains at least $2^k$ vertex-disjoint induced copies of $S$, and hence has at least $16\cdot2^k$ vertices.
\end{proof}

The two copies in the proof may have edges between them. Vertex-disjointness does not require their union to be an induced disjoint union. In particular, $P$ and $Q$ themselves share edges in $S$ and are adjacent vertices of $\Cy5(S)$; this does not give an edge belonging to both induced copies.

\section{The trichotomy theorem}
The following result is the specialization of the fundamental theorem on graph operators~\cite{fundamental} to $\Cy5$. We include a proof that explicitly handles the distinction between an unbounded sequence and one tending to infinity.

\begin{theorem}[Trichotomy]\label{thm:trichotomy}
Every finite simple graph is exactly one of pentagon-vanishing, pentagon-periodic, and pentagon-expanding.
\end{theorem}
\begin{proof}
Write $G_k=\Cy5^k(G)$. If some $G_k$ is empty, all subsequent iterates are empty, so $G$ is pentagon-vanishing. Suppose instead that every $G_k$ is nonempty.

If $|V(G_k)|\to\infty$, the graph is pentagon-expanding. Otherwise, there is an integer $N$ such that $|V(G_k)|\le N$ for infinitely many $k$. There are only finitely many isomorphism classes of finite simple graphs of order at most $N$. Therefore $G_i\cong G_j$ for some $i<j$. Since $\Cy5$ preserves isomorphism,
\[
G_{i+t}\cong G_{j+t}\qquad(t\ge0).
\]
Thus $G$ is pentagon-periodic, with an eventual period dividing $j-i$.

A periodic sequence of graph isomorphism classes has bounded orders, including its finite initial segment, and therefore cannot be expanding. Vanishing graphs satisfy neither of the other definitions. The three alternatives are consequently exhaustive and mutually exclusive.
\end{proof}

\section{Concluding remarks and further directions}
The one-vertex extension theorem identifies an exact boundary between eventual fixed-point behavior and expansion around the icosahedron: a new induced pentagon occurs if and only if the extension is expanding. This strengthens the tadpole-hat construction from a sufficient condition for one attachment type to a classification covering every possible neighborhood of the added vertex. The proof also explains the transition: every new pentagon produces a tadpole hat within at most two iterations.

The Shrikhande construction gives a different mechanism. Two disjoint translation families of pentagons each induce a copy of the original graph. Proposition~\ref{prop:replication} turns this finite construction into an exponential bound for all iterates. In particular, every graph containing an induced copy of the Shrikhande graph is pentagon-expanding and satisfies the same bound $16\cdot2^k$. Corollary~\ref{cor:attachment-supergraph} similarly transfers the icosahedral attachment result to larger graphs.

Several questions remain. Which other connected graphs are fixed points of $\Cy5$, and which periods greater than one occur? Can the exponential bounds in Theorems~\ref{thm:hats}, \ref{thm:attachment}, and~\ref{thm:shrikhande} be improved? How does the classification change when several vertices are added to an icosahedron, particularly when new pentagons use two or more outside vertices? More generally, for the operator defined by edge intersections of induced cycles of length $\ell\ge3$, which attachments or collections of induced subgraphs force expansion? The replication principle applies at every such length, but finding the required induced copies remains a separate structural problem.


\begin{thebibliography}{99}
\bibitem{Balakrishnan} R. Balakrishnan, \emph{Triangle graphs}, in: Graph connections (Cochin, 1998), p.~44, Allied Publ., New Delhi, 1999.
\bibitem{EgawaRamos} Y. Egawa, R.~E. Ramos, \emph{Triangle graphs}, Math. Japon. 36 (1991) pp. 465--467.	
\bibitem{Egawa} Y. Egawa, M. Kano, E.~L. Tan, \emph{On cycle graphs}, Ars Combin.  \textbf{32} (1991) pp. 97--113. \href{https://combinatorialpress.com/ars-articles/volume-032-ars-articles/on-cycle-graphs/}{https://combinatorialpress.com/ars-articles/volume-032-ars-articles/on-cycle-graphs/}
\bibitem{fundamental}  S.~V. Gervacio, \emph{A fundamental theorem on graph operators},  Mathematics Open, Vol.~4, 2550010 (2025), \url{https://doi.org/10.1142/S2811007225500105}
\bibitem{quadrangle} S.~V. Gervacio and Y.~F. Lim, \emph{The quadrangle graph operator}, Mathematics Open, Vol.~5, 2550019 (2026) \url{https://doi.org/10.1142/S2811007225500191}
\bibitem{gervacio} S.~V. Gervacio, \emph{Cycle graphs}. In K.~M.~Koh \& H.~P.~Yap (Eds.), Proceedings of the First Southeast Asian Graph Theory Colloquium, (1984) pp. 279--293. Singapore: Springer-Verlag. \href{https://link.springer.com/chapter/10.1007/BFb0073127}{https://link.springer.com/chapter/10.1007/BFb0073127}
\bibitem{seema} S. Varghese, S.~V. Gervacio, A. Vijayakumar, \emph{On cycle graphs}, Asian-European Journal of Mathematics, 2023, \href{https://doi.org/10.1142/S1793557123500365}{https://doi.org/10.1142/S1793557123500365}
	
\bibitem{shrikhande} S.~S. Shrikhande, \emph{The uniqueness of the $L_2$ association scheme}, The Annals of Mathematical Statistics \textbf{30} (3) (1959), 781--798. \url{https://doi.org/10.1214/aoms/1177706207}
\end{thebibliography}
\end{document}